\documentclass[11pt,reqno]{amsart}
\usepackage{amsmath}
\usepackage{cases}
\usepackage{mathrsfs}
\usepackage{bbm}
\usepackage{amssymb}
\usepackage{amscd}
\usepackage{amsfonts,latexsym,amsmath,amsthm,amsxtra,mathdots,amssymb,latexsym,mathabx}
\usepackage[all,cmtip]{xy}
\RequirePackage{amsmath} \RequirePackage{amssymb}
\usepackage{color}
\usepackage{colordvi}
\usepackage{multicol}
\usepackage{hyperref}
\usepackage{mathtools}
\usepackage[margin=1in]{geometry}
\usepackage{xcolor}
\hypersetup{
	colorlinks,
	linkcolor={red!50!black},
	citecolor={blue!50!black},
	urlcolor={blue!80!black}
}
\usepackage{cite}

\newcommand{\bea}{\begin{eqnarray}}
	\newcommand{\eea}{\end{eqnarray}}
\newcommand{\bna}{\begin{eqnarray*}}
	\newcommand{\ena}{\end{eqnarray*}}

\numberwithin{equation}{section}

\theoremstyle{plain}
\newtheorem*{Theorem A}{Theorem A}
\newtheorem*{Theorem B}{Theorem B}
\newtheorem{problem}{Problem}[section]
\newtheorem{lemma}{Lemma}[section]
\newtheorem{theorem}{Theorem}[section]

\theoremstyle{definition}

\allowdisplaybreaks[3]

\begin{document}
	
	\title
	[{On a problem of Pongsriiam on the sum of divisors}] {On a problem of Pongsriiam on the sum of divisors}
	
	\author
	[R.-J. Wang] {Rui-Jing Wang}
	
	\address{(Rui-Jing Wang) School of Mathematical Sciences,  Jiangsu Second Normal University, Nanjing 210013, People's Republic of China}
	\email{\tt wangruijing271@163.com}
	
	
	

	\begin{abstract}
		For any positive integer $n$, let $\sigma (n)$ be the sum of all
		positive divisors of $n.$ In this paper, it is proved that for every integer $ 1\leq k\leq 29,\ (k,30)=1, $ we have $$\sum_{n\leq K}\sigma(30n)>\sum_{n\leq K}\sigma(30n+k)$$
		for all  $K\in \mathbb{N},$  which gives a positive answer to a problem posed by Pongsriiam recently.
	\end{abstract}
	
	\maketitle
	\section{Introduction}
	For any positive integer $n$, let $\sigma (n)$ be the sum of all
	positive divisors of $n.$ In this paper, we always assume that $ x $ is a real number, $ m $ and $ n $ are positive integers, $ p $ is a prime, $ p_{n} $ is the $ n $-th
	prime, $ \phi(n) $ denotes the Euler totient function.
	Jarden\cite[p.65]{Jarden} observed that $ \phi(30n+1)>\phi(30n) $ for all $ n\leq 10^{5} $ and later the inequality was calculated to be true up to $ 10^9. $ However, Newman\cite{Newman} proved that there are infinitely many $ n $ such that $ \phi(30n+1)<\phi(30n) $ and the smallest one is $$ \frac{p_{385}p_{388}\prod_{j=4}^{383}p_{j}-1}{30},$$ which was given by Martin\cite{Martin}. For related research, one may refer
	to \cite{Erdos1936,Luca,Pollack,Kobayashi2020,Wang}. It is certainly natural to consider the analogous problem for the
	sum of divisors function. Recently, Pongsriiam\cite[Theorem 2.4]{Pongsriiam} 
	proved that $\sigma(30n)-\sigma(30n+1) $ also has infinitely many sign changes. He found that 
	$\sigma(30n)>\sigma(30n+1)$ for all $ n\leq 10^7 $ and posed the following relevant problem.
	\begin{problem}\label{p2}(\cite[Problem 3.8(ii)]{Pongsriiam}~)~
		Is it true that $$\sum_{n\leq K}\sigma(30n)>\sum_{n\leq K}\sigma(30n+1)$$
		for all $ K\in \mathbb{N}? $
	\end{problem}
	Recently, Ding, Pan and Sun\cite{Ding} solved several problems of Pongsriiam. Inspired by their ideas, we answer affirmatively the above Problem \ref{p2}.
	In fact, we prove a slightly stronger result.
	\begin{theorem} \label{thm1} For every integer $ 1\leq k\leq 29,\ (k,30)=1, $ we have
		$$\sum_{n\leq K}\sigma(30n)>\sum_{n\leq K}\sigma(30n+k)$$
		for all  $K\in \mathbb{N}.$
	\end{theorem}

	\section{Estimations}
	Let $$ \beta_{0}=\sum_{d=1}^{\infty}\frac{B_{0}(d)}{d^2},$$
	where $B_{0}(d)$ denotes the number of solutions, not counting multiplicities, of the congruence $30m\equiv 0 \pmod d. $ For every integer $ 1\leq k\leq 29,\ (k,30)=1, $ let $$ \beta_{k}=\sum_{d=1}^{\infty}\frac{B_{k}(d)}{d^2},$$
	where $B_{k}(d)$ denotes the number of solutions, not counting multiplicities, of the congruence $30m+k\equiv 0 \pmod d. $ By the
	Chinese reminder theorem, both $ B_{0}(d)$ and $ B_{k}(d) $ are multiplicative. Note that for $ p\nmid 30 $ we have $ B_{0}(p^{\alpha})=1$ and $ B_{k}(p^{\alpha})=1$ for any positive integer $ \alpha $. It is obvious that $ B_{0}(p^{\alpha})=p$ and $ B_{k}(p^{\alpha})=0$ for $ p=2,3,5 $ and any positive integer $ \alpha $. It follows that 
	$ B_{0}(d)\leq 30$ and $ B_{k}(d)\leq 1 $ for any positive integer $ d. $
	By \cite[Theorem 11.7]{Apostol} and 
	$$ \frac{\pi^2}{6}=\zeta(2)=\prod_{p}\left(1-p^{-2}\right)^{-1}, $$ 
	we have $$
	\beta_{0}=\prod_p\left(1+\frac{B_{0}(p)}{p^2}+\frac{B_{0}(p^2)}{p^4}+\cdots\right)=\frac{5}{3} \frac{11}{8} \frac{29}{24} \prod_{p \nmid 30}\left(1-p^{-2}\right)^{-1}=\frac{319\pi^2}{1080}
	$$
	and
	$$
	\beta_{k}=\prod_p\left(1+\frac{B_{k}(p)}{p^2}+\frac{B_{k}(p^2)}{p^4}+\cdots\right)=\prod_{p \nmid 30}\left(1-p^{-2}\right)^{-1}=\frac{8\pi^2}{75}.
	$$
	
	We always assume that $ x\geq 1000, $ $ 1\leq k\leq 29,\ (k,30)=1 $ in the following Lemmas.
	
	It could be checked that there would be large oscillations of the error terms which influence the main terms if one tries to calculate directly those sums asked in problem \ref{p2}. Therefore, we manipulate the weighted sums firstly and then transform them to the original sums via summations by parts. 
	\begin{lemma}\label{lem1b} We have
		$$\sum_{m \leq x} \frac{\sigma(30m)}{30m}=\beta_{0} x+g(x),$$
		where $$-30\log 30x-32< g(x)<30\log 30x +32.$$
	\end{lemma}
	
	\begin{proof} 
		By the definition of $ B_{0}(d), $
		\begin{eqnarray*}
			\sum_{m \leq x} \frac{\sigma(30m)}{30m} 
			&=&\sum_{m \leq x} \sum_{d\mid 30m} \frac{1}{d} \\
			&=&\sum_{d \leq 30x} \frac{1}{d} \sum_{m \leq x \atop
				30m\equiv 0\hskip -2.5mm\pmod d} 1\\
			&=&\sum_{d \leq 30x} \frac{B_{0}(d)}{d}\left(\frac{x}{d}+\alpha_{0}(x,d)\right) \\
			&=&x \sum_{d \leq 30x} \frac{B_{0}(d)}{d^2}+\sum_{d \leq 30x} \frac{B_{0}(d)}{d}\alpha_{0}(x,d) \\
			&=&\beta_{0} x-x \sum_{d>30x} \frac{B_{0}(d)}{d^2}+\sum_{d \leq 30x} \frac{B_{0}(d)}{d}\alpha_{0}(x,d) \\
			&=&\beta_{0} x+g(x),
		\end{eqnarray*}
		where $ -1\leq \alpha_{0}(x,d)\leq 1 $ and 
		$$g(x)= -x \sum_{d>30x} \frac{B_{0}(d)}{d^2}+\sum_{d \leq 30x} \frac{B_{0}(d)}{d}\alpha_{0}(x,d).$$
		By \cite[Theorem 3.2]{Apostol},
		$$ 0\leq \sum_{d>30x} \frac{B_{0}(d)}{d^2}\leq 30\sum_{d>30x} \frac{1}{d^2}\leq 30\left(\frac{1}{30x}+\frac{30x-[30x]}{(30x)^2}\right) $$
		and 
		$$ 0\leq \sum_{d\leq 30x} \frac{B_{0}(d)}{d}\leq 30\sum_{d\leq 30x} \frac{1}{d}<30(\log 30x +1). $$
		It follows that
		$$ -30\log 30x-32< g(x)<30\log 30x +32. $$
		
		This completes the proof of Lemma \ref{lem1b}.
	\end{proof}

	\begin{lemma}\label{lem1c} We have
		$$\sum_{m \leq x} \frac{\sigma(30m+k)}{30m+k}=\beta_{k} x+h_{k}(x),$$
		where $$-\log (30x+k) -2< h_{k}(x)<\log (30x+k) +2.$$
	\end{lemma}
	
	\begin{proof} 
		By the definition of $ B_{k}(d), $
		\begin{eqnarray*}
			\sum_{m \leq x} \frac{\sigma(30m+k)}{30m+k} 
			&=&\sum_{m \leq x} \sum_{d\mid 30m+k} \frac{1}{d} \\
			&=&\sum_{d \leq 30x+k} \frac{1}{d} \sum_{m \leq x \atop
				30m+k\equiv 0\hskip -2.5mm\pmod d} 1\\
			&=&\sum_{d \leq 30x+k} \frac{B_{k}(d)}{d}\left(\frac{x}{d}+\alpha_{k}(x,d)\right) \\
			&=&x \sum_{d \leq 30x+k} \frac{B_{k}(d)}{d^2}+\sum_{d \leq 30x+k} \frac{B_{k}(d)}{d}\alpha_{k}(x,d) \\
			&=&\beta_{k} x-x \sum_{d>30x+k} \frac{B_{k}(d)}{d^2}+\sum_{d \leq 30x+k} \frac{B_{k}(d)}{d}\alpha_{k}(x,d) \\
			&=&\beta_{k} x+h_{k}(x),
		\end{eqnarray*}
		where $ -1\leq \alpha_{k}(x,d)\leq 1 $ and 
		$$h_{k}(x)= -x \sum_{d>30x+k} \frac{B_{k}(d)}{d^2}+\sum_{d \leq 30x+k} \frac{B_{k}(d)}{d}\alpha_{k}(x,d)$$
		By \cite[Theorem 3.2]{Apostol},
		$$ 0\leq \sum_{d>30x+k} \frac{B_{k}(d)}{d^2}\leq \sum_{d>30x+k} \frac{1}{d^2}\leq \frac{1}{30x+k}+\frac{30x+k-[30x+k]}{(30x+k)^2} $$
		and 
		$$ 0\leq \sum_{d\leq 30x+k} \frac{B_{k}(d)}{d}\leq \sum_{d\leq 30x+k} \frac{1}{d}<\log (30x+k) +1. $$ 
		It follows that
		$$ -\log (30x+k) -2< h_{k}(x)<\log (30x+k) +2. $$
		
		This completes the proof of Lemma \ref{lem1c}.
	\end{proof}
	
	\begin{lemma}\label{lem2b} We have
		$$\sum_{m \leq x} \sigma(30m)>15\beta_{0} x^{2}-2000x\log 30x.$$
	\end{lemma}
	
	\begin{proof} 
		Let $$ S(x)=\sum_{m \leq x} \frac{\sigma(30m)}{30m}. $$
		By \cite[Theorem 3.1]{Apostol},
		\begin{eqnarray*}
			\sum_{m \leq x} \sigma(30m)
			&=&\sum_{m \leq x} 30m(S(m)-S(m-1))\\
			&=&30[x]S([x])-\sum_{m \leq x-1} (30(m+1)-30m)S(m)-30S(0)\\
			&>&30(x-1)(\beta_{0} x-30\log 30x -32)-30\sum_{m \leq x-1}(\beta_{0} m+30\log 30m +32) \\
			&>&30\beta_{0} x^{2}-900x\log 30x -960x-15\beta_{0} x^{2}-900x\log 30x-960x\\
			&&\quad-30\beta_{0} x+900\log 30x+960\\
			&>&15\beta_{0} x^{2}-2000x\log 30x.
		\end{eqnarray*}
		
		This completes the proof of Lemma \ref{lem2b}.
	\end{proof}
	
	\begin{lemma}\label{lem2c} We have
		$$\sum_{m \leq x} \sigma(30m+k)<15\beta_{k} x^{2}+100x\log (30x+k).$$
	\end{lemma}
	
	\begin{proof} 
		Let $$ T_{k}(x)=\sum_{m \leq x} \frac{\sigma(30m+k)}{30m+k}. $$
		By \cite[Theorem 3.1]{Apostol},
		\begin{eqnarray*}
			\sum_{m \leq x} \sigma(30m+k)
			&=&\sum_{m \leq x} (30m+k)(T_{k}(m)-T_{k}(m-1))\\
			&=&(30[x]+k)T_{k}([x])-\sum_{m \leq x-1} (30(m+1)+k-30m-k)T_{k}(m)\\
			&<&(30x+k)(\beta_{k} x+\log (30x+k) +2)-30\sum_{m \leq x-1}(\beta_{k} m-\log (30m+k) -2) \\
			&<&30\beta_{k} x^{2}+30x\log (30x+k) +60x-15\beta_{k} (x-2)^{2}+30x\log (30x+k)\\
			&&\quad+60x+k\beta_{k} x+k\log (30x+k) +2k\\
			&<&15\beta_{k} x^{2}+100x\log (30x+k).
		\end{eqnarray*}
		
		This completes the proof of Lemma \ref{lem2c}.
	\end{proof}

\section{Proof of Theorem \ref{thm1}}

\begin{proof}[Proof of Theorem \ref{thm1}] For every integer $ 1\leq k\leq 29,\ (k,30)=1, $ by Lemmas \ref{lem2b} and \ref{lem2c}, we have $$\sum_{m \leq x} \sigma(30m)>15\beta_{0} x^{2}-2000x\log 30x>15\beta_{k} x^{2}+100x\log (30x+k)>\sum_{m \leq x} \sigma(30m+k)$$
provided that $ x\geq 1000. $ For every integer $ 1\leq k\leq 29,\ (k,30)=1, $ it is easy to verify that $$\sum_{m \leq K} \sigma(30m)>\sum_{m \leq K} \sigma(30m+k)$$
for every positive integer $ K<1000 $ by programming.

This completes the proof of Theorem \ref{thm1}.
\end{proof}

	\bigskip
	
	\section*{Acknowledgments}
	The author would like to thank the referee and Yuchen Ding for their helpful suggestions.

\end{document}